\documentclass[12pt,leqno]{article}
\usepackage[shortlabels]{enumitem}
\usepackage{preamble}

\title{A potential theory for the Wess--Zumino--Witten equation in the space of K\"ahler potentials}
\author{Kuang-Ru Wu}

\newcommand{\RN}[1]{%
  \textup{\uppercase\expandafter{\romannumeral#1}}%
}

\usepackage{fullpage}
\usepackage{multicol}
\usepackage{fancyhdr}

\theoremstyle{plain}

\newtheorem{conjecture}[theorem]{Conjecture}
\numberwithin{equation}{section}

\begin{document}

\date{}

\parskip=6pt

\maketitle

\begin{abstract}
We develop a potential theory for the Wess--Zumino--Witten (WZW) equation in the space of K\"ahler potentials which is parallel to the potential theory for the Hermitian--Yang--Mills equation. A concept called $\omega$-harmonicity on graphs is introduced which characterizes the WZW equation. We also show that, with respect to a Banach--Mazur type distance function, the distance between two solutions of the WZW equation is subharmonic. 

The harmonic map into the space of K\"ahler potentials, as a special case of the WZW equation, is also investigated. In particular, we show the solvability of the Dirichlet problem for the harmonic map, and the approximation/quantization by its finite dimensional counterparts.   
\end{abstract}

\section{Introduction}

Let $X$ be a compact complex manifold of dimension $n$ with a K\"ahler form $\omega$. The space of K\"ahler potentials is $\mathcal H_{\omega}=\{\phi\in C^\infty(X,\mathbb{R}): \omega_\phi=\omega +i\partial\Bar{\partial }\phi>0 \}$. In the study of canonical metrics in K\"ahler geometry, the geodesic equation in $\mathcal{H}_\omega$ is indispensable (\cite{Semmesmonge,Donaldson99,XXChen}). One generalization of the geodesic equation is the Wess--Zumino--Witten (WZW) equation. For a map $\Phi$ from an open subset $ D\subset \mathbb{C}^m$ to $ \mathcal H_\omega$, the WZW equation is 
\begin{equation}\label{WZW}  
\sum^m_{j=1}|\nabla\Phi_{z_j}|_{\Phi}^2-2\Phi_{z_j\bar{z}_j}+i\{\Phi_{\bar{z}_j},\Phi_{{z_j}}\}_{\Phi}=0,  
\end{equation} 
where $\{z_j\}$ are coordinates on $D$ and $|\nabla \Phi_{z_j}(z)|_{\Phi}^2$ is computed using the metric $\omega_{\Phi(z)}$, and $\{\cdot ,\cdot \}_{\Phi(z)}$ is the Poisson bracket determined by the symplectic form $\omega_{\Phi(z)}$. An equivalent way of writing the WZW equation is, viewing $\Phi$ as a function $D\times X \to \mathbb{R}$,
$$(i\partial \bar{\partial}\Phi+\pi^*\omega)^{n+1}\wedge(i\sum_{j=1}^m dz_j\wedge d\bar{z}_j)^{m-1} = 0,$$
where $\pi$ is the projection from $D\times X$ to $X$.
For motivation and application of the WZW equations, see \cite{Semmesmonge,Donaldson99,RubinsteinZelditch,LempertLegendre,Wu23, finski2024lower, finski2024harder}.

Another generalization of the geodesic equation is the harmonic map equation. For an open set $D$ in $\mathbb{C}^m$ with the Euclidean metric and $\mathcal{H}_\omega$ with the Mabuchi metric, the harmonic map equation for $\Phi:D\to \mathcal H_\omega$ is the following  
\begin{equation}\label{har}
\sum^m_{j=1}|\nabla\Phi_{z_j}|_{\Phi}^2-2\Phi_{z_j\bar{z}_j}=0.    
\end{equation} We will study the WZW equations first and then apply the results to the harmonic map equations. In fact, we will consider the more general case where the harmonic map equations have domains in $\mathbb{R}^m$ (see Theorem \ref{thm weak harmonic}).

In \cite{Wu23}, we show the existence of weak solutions to the Dirichlet problem for the WZW equation by interpolation. One purpose of this paper is to understand the weak solution through potential theory. Since the WZW equations can be approximated by the Hermitian--Yang--Mills (HYM) equations (\cite[Theorem 1.2]{Wu23}) and the potential theory for HYM equations was established already by Rochberg, Slodkowski, Coifman and Semmes in \cite{Rochberg,SlodI,CoifmanSemmes}, it is natural to expect some potential theory for the WZW equations. In this paper, we develop such a potential theory that is almost analogous to the HYM potential theory. 

Let us recall the following definition introduced in \cite{Wu23} 
\begin{definition}\label{sub graphs}
 An upper semicontinuous function $u:D\times X\to [-\infty, \infty)$ is said to be $\omega$-subharmonic on graphs if, for any holomorphic map $f$ from an open subset of $D$ to $X$, $\psi(f(z))+u(z,f(z))$ is subharmonic, where $\psi$ is a local potential of $\omega$.
\end{definition}
In \cite{Wu23}, we did not include upper semicontinuity in the definition. We do so in this paper to make the presentation cleaner. Definition \ref{sub graphs} corresponds to the subsolution of the WZW equation (see \cite[Lemma 4.1]{Wu23} or Lemma \ref{+ det} below). To study the supersolution, we introduce the idea of $\omega$-superharmonicity on graphs.

\begin{definition}\label{super}
    A lower semicontinuous function $u:D\times X\to (-\infty, \infty]$ is said to be $\omega$-superharmonic on graphs if for any $v$ that is $\omega$-subharmonic on graphs in $U\times X$ with $U$ an open subset of $D$, the function $h(z):=\sup_{x\in X}v(z,x)-u(z,x)$ is subharmonic in $U$. 
\end{definition}
A function $u$ is said to be $\omega$-harmonic on graphs if it is both $\omega$-subharmonic and $\omega$-superharmonic on graphs. These definitions are motivated by the study of interpolation problems in \cite{Rochberg,SlodI,SlodII,SlodIII,RevSemmes,CoifmanSemmes}. The idea of $\omega$-harmonicity on graphs is exactly what we need in order to develop a potential theory for the WZW equations, and it greatly refines the results in \cite{Wu23}. Especially, the solutions to the WZW equation can be characterized by $\omega$-harmonicity on graphs: 
\begin{theorem}\label{characterize}
Suppose $u$ is a $C^2$ function on $D\times X$ and $\omega +i\partial\Bar{\partial}u(z,\cdot)>0$ on $X$ for all $z\in D$. Then $u$ is $\omega$-harmonic on graphs if and only if $u$ solves the WZW equation
$$(i\partial \bar{\partial}u+\pi^*\omega)^{n+1}\wedge(i\sum_{j=1}^m dz_j\wedge d\bar{z}_j)^{m-1} = 0.$$
\end{theorem}
Thus, the study of the WZW equations can be translated to the study of $\omega$-harmonicity on graphs. We would like to emphasize that such a translation from partial differential equations to the potential theory is made possible because of our introduction of $\omega$-harmonicity on graphs, and this is one of the novelties in this paper.    

In the theory of interpolation of norms \cite{Rochberg,SlodI,CoifmanSemmes}, one can pass between subharmonicity and superharmonicity by taking the dual of norms (such duality is proved by Slodkowski in \cite{SlodI,SlodIII}). However, in our case of $\omega$-sub/superharmonicity, the duality is still missing. One possibility is the Legendre transform (\cite{legendredualBCEKR}, see also \cite{LLmathann83,LLduke85,complexinterpolationBCEKR}). We provide some justification and make a conjecture at the end of Section \ref{section charac}.

The concept of $\omega$-harmonicity on graphs is also linked to the geometry of $\mathcal H_\omega$. The space of K\"ahler potentials $\mathcal{H}_\omega$ with the Mabuchi metric is non-positively curved (\cite{Mabuchi,Donaldson99,CalabiChen,ChenSun}), so the distance between two harmonic maps into $\mathcal H_\omega$ is subharmonic (\cite[Chapter II.1 and 2]{CAT0book}). The WZW equations which differ from the harmonic map equations by the Poisson bracket have a similar property, but with a different distance function defined as follows. For two functions $u$ and $v$ on $D\times X$, we define 
\begin{equation}\label{dist}
  d(u,v)=\max (\delta(u,v),\delta(v,u))\,\, \text{  where } \delta(u,v)(z):=\sup_{x\in X} u(z,x)-v(z,x).  
\end{equation}
This distance function is a variant of the Banach-Mazur type distance introduced in \cite[Page 160]{RevSemmes} and \cite[Formulas (8.1) and (8.2)]{CoifmanSemmes}. With respect to (\ref{dist}), we have
\begin{theorem}\label{dist subhar}
    If $u$ and $v$ are $\omega$-harmonic on graphs in $D\times X$, then the distance function $z\mapsto d(u,v)$ is subharmonic in $D$.
\end{theorem}
The proof is a direct consequence of the definition. Indeed, since $v$ is $\omega$-superharmonic on graphs and $u$ is $\omega$-subharmonic on graphs, the function $\delta(u,v)(z)$ is subharmonic in $D$ by Definition \ref{super}. Switching $u$ and $v$, we see that $\delta(v,u)(z)$ is also subharmonic in $D$, and hence $d(u,v)$ is subharmonic in $D$.

According to Theorem \ref{characterize}, functions $\omega$-harmonic on graphs can be viewed as weak solutions for the WZW equations, so Theorem \ref{dist subhar} basically says that the distance between two solutions for the WZW equations is subharmonic.

Let us apply the potential theory above to the Dirichlet problem for the WZW equation. We recall the interpolation of the Dirichlet problem first. Let $D$ be a regular bounded open set in $\mathbb{C}^m$. Let $\nu$ be a continuous map from $\partial D\times X $ to $ \mathbb{R}$ such that $\nu(z,\cdot)\in \text{PSH}(X,\omega)$ for $z\in \partial D$, and denote by $V$ the sup of the Perron family 
\begin{align*}
G_\nu:=\{u \in\textup{usc($D\times X$)}:& \textup{ $u$ is $\omega$-subharmonic on graphs}, \textup{and }\limsup_{D \ni z \to \zeta\in \partial D}u(z,x)  \leq \nu(\zeta,x)\}.
\end{align*} 
We assume here $\omega$ is in an integral class, so that by \cite{Wu23} $V$ is continuous, attains the boundary data $\nu$, and it is a weak solution in the sense that if $V$ is $C^2$ then it solves the WZW equation. (It is likely that the integral assumption on $\omega$ can be removed, but we have not been able to do so. Note also that although in \cite{Wu23} the boundary data is assumed in $C^\infty(\partial D, \mathcal H_\omega)$, the same results still hold for continuous boundary data; the only concern is the use of Tian--Catlin--Zelditch asymptotic theorem in line 14 on page 350 in \cite{Wu23}, but the same estimate can also be obtained by \cite[Lemma 5.10]{DarvasWu}. Yet another improvement from \cite{Wu23} is that the boundary of $D$ is relaxed from strongly pseudoconvex to regular. See Lemma \ref{barrier}). We give a more precise description of $V$ here.
\begin{theorem}\label{sup is har}
    The sup $V$ of the Perron family $G_\nu$ is $\omega$-harmonic on graphs.
\end{theorem}
Theorem \ref{sup is har} gives the existence part for the Dirichlet problem. Uniqueness follows readily from Theorem \ref{dist subhar}. Indeed, if $V_1$ and $V_2$ are $\omega$-harmonic on graphs in $D\times X$ and assume the boundary data $\nu$, then $V_1=V_2$ in $D\times X$ by Theorem \ref{dist subhar} and the maximum principle that ensues. 

We mainly focus on the trivial fibration $D\times X$ in this paper. For the developments of nontrivial fibration, see for example \cite{CampanaCaoMihai, finski2024lower, finski2024harder, wu2025mean}.

\subsection{Harmonic maps}\label{subsection har}
In this subsection, we will apply the results developed above and in \cite{Wu23} to study harmonic maps into the space of K\"ahler potentials $\mathcal{H}_\omega$ or more generally into $\text{PSH}(X,\omega)$.

We assume that $L$ is a positive line bundle over $X$, and $h$ is a positively curved metric on $L$ whose curvature equals $\omega$ (so $\omega$ is in an integral class). For a positive integer $k$, we denote by $\mathcal H_k$ the space of inner products on $H^0(X,L^k)$.

We will show the existence of a weak solution to the Dirichlet problem for harmonic maps into the space of K\"ahler potentials $\mathcal{H}_\omega$; moreover, such a solution can be approximated by harmonic maps into $\mathcal H_k$. (Rubinstein and Zelditch in \cite[Theorem 1.1]{RubinsteinZelditch} solved the Dirichlet problem and showed $C^2$ approximation, but their results are for the toric case.)

Let $D'$ be a regular bounded open set in $\mathbb{R}^m$. We will denote the coordinates in $D'$ by $t=(t_1,\dots, t_m)\in  \mathbb{R}^m$. Let $\nu$ be a real-valued continuous function on $\partial D'\times X$ such that for fixed $t\in \partial D'$ the function $ \nu(t,\cdot)$ on $X$ is in $ \text{PSH}(X,\omega)$.

Let us first consider the approximation part and the harmonic maps into $\mathcal H_k$. To that end, we will use the Hilbert map $H_k:\text{PSH}(X,\omega) \to \mathcal H_k$ and the Fubini--Study map $FS_k: \mathcal H_k\to \mathcal H_{\omega}$ (which will be recalled in Section \ref{section harmonic}). For the boundary data $H_k(\nu):\partial D'\to \mathcal H_k$, the following Dirichlet problem for the harmonic map equation 
\begin{equation}
\begin{dcases}\label{harmonic matrix}
 \sum_{j=1}^m \frac{\partial}{\partial t_j}\big((V^k)^{-1}  \frac{\partial V^k}{\partial t_j}\big)=0 \\
V^k|_{\partial D'}=H_k(\nu)
\end{dcases}    
\end{equation}
has a unique solution $V^k:\overline{D'}\to \mathcal H_k $ which is continuous on $\overline{D'}$ and smooth on $D'$ by \cite{Hamilton}.
The metric on $D'$ is Euclidean, and the metric on $\mathcal{H}_k$ is $(A,B)=\text{Tr}(h^{-1}Ah^{-1}B)$ for $h\in\mathcal{H}_k$ and $A,B\in T_h \mathcal H_k$ (see \cite[Section 1, Chapter VI]{Kobabook}).

\begin{theorem}\label{thm weak harmonic} Let $\nu $ be a continuous map from $\partial D'\times X$ to $\mathbb{R}$ such that $\nu(t, \cdot)\in \text{PSH}(X,\omega)$ for $t\in \partial D'$.
There exists a continuous function $V$ on $\overline{D'}\times X$ such that for fixed $t\in \overline{D'}$ the function $V(t,\cdot)$ is in $ \text{PSH}(X,\omega)$, $V|_{\partial D'}=\nu$, and $V$ can be approximated/quantized by harmonic maps $V^k:\overline{D'}\to \mathcal H_k$ from (\ref{harmonic matrix}) in the following sense. The functions $FS_k(V^k)$ converge to $V$ uniformly on $D'\times X$ as $k\to \infty$. Moreover, if $V$ is $C^2$, then it solves the harmonic map equation
\begin{equation}\label{1.5}
\sum^m_{j=1}|\nabla V_{t_j}|_{V}^2-2V_{t_jt_j}=0.
\end{equation}
\end{theorem}
We will view $V$ as a weak solution of the harmonic map equations. The idea of the proof is based on the fact that the harmonic map equation can be complexified to the WZW equation just like the geodesic equation can be complexified to the Monge--Amp\`ere equation. To be precise, we complexify $t_j\in \mathbb{R}$ by considering $e^{t_j+\sqrt{-1}s_j}=\zeta_j\in \mathbb{C}$ and extend the boundary data $\nu$ to a rotationally invariant boundary data $\tilde{\nu}$. The sup of the Perron family $G_{\tilde{\nu}}$ will be $\omega$-harmonic on graphs (by Theorem \ref{sup is har}) and rotationally invariant (by Theorem \ref{dist subhar} and the maximum principle), and hence it defines a function in $D'\times X$ which is the $V$ in Theorem \ref{thm weak harmonic} (see Section \ref{section harmonic} for details). 

Let us compare Theorem \ref{thm weak harmonic} with previous results and indicate our new input. The geodesic approximation in $\mathcal{H}_\omega$ is first proved by Phong and Sturm in \cite{PhongSturm} and refined by Berndtsson in \cite{berndtsson2013probability} with $C^0$ approximation. The harmonic approximation is proved by Rubinstein and Zelditch in \cite{RubinsteinZelditch} with $C^2$ approximation but for the toric case. In \cite{Wu23}, we prove the approximation of the WZW equation by the HYM equations which covers the results in \cite{PhongSturm,berndtsson2013probability, RubinsteinZelditch} as special cases. However, it was not known how to prove the most general harmonic approximation (without the toric assumption) using the results in \cite{Wu23}. It is the idea of $\omega$-harmonicity on graphs that enables us to prove the most general harmonic approximation (Theorem \ref{thm weak harmonic}), and this is the new input in this paper.

When $m=1$ and $D'$ is the open interval $(0,1)$, $V$ of Theorem \ref{thm weak harmonic} defines a geodesic in $\text{PSH}(X,\omega)$. Actually, such $V$ behaves also like a geodesic with respect to the distance $d$ in (\ref{dist}) in the sense that  
\begin{equation}\label{proportional}
 d(V(s),V(t))=|s-t|d(V(0),V(1))   
\end{equation}
(the proof is given in Section \ref{sec appendix}. Such a formula is already mentioned in \cite{RevSemmes,CoifmanSemmes} for interpolation of norms. That the same curve behaves like a geodesic with respect to different metrics is discussed in \cite{Tamasadv,DLR,LLleastaction}). However, the curve $t\mapsto tV(1)+(1-t)V(0)$ also satisfies (\ref{proportional}) as one can check easily. For higher dimension $m$, the harmonic maps can probably be used to find flats in $\mathcal H_\omega$ (\cite{remi2023infinite, darvas2025lines}).

Besides the harmonic maps or the HYM equations above, one may also think about twisted harmonic maps or the Hitchin equation on the bundle $D\times H^0(X, L^k)\to D$ (\cite{Hitchin,Donaldsontwisted,Corlette,Simpson}). But we do not know what the 'limiting' equation in $D\to \mathcal{H}_\omega$ should look like. We hope to investigate the Hitchin equation in the space of K\"ahler potentials and its quantization in a future paper.

We organize the paper as follows. We prove four lemmas in Section \ref{sec lemma} and prove the characterization of the WZW equation, Theorem \ref{characterize}, in Section \ref{section charac}. Remarks about Legendre duality and $\omega$-harmonicity are also given in Section \ref{section charac}. The fact that the sup of the Perron family is $\omega$-harmonic is proved in Section \ref{section sup is har}. The harmonic maps into $\mathcal H_\omega$ and their quantization are discussed in Section \ref{section harmonic}. In Section \ref{sec appendix}, we show that the geodesics in $\mathcal H_\omega$ with respect to the Mabuchi metric also behave like geodesics with respect to the distance $d$ in (\ref{dist}).

I would like to thank R\'emi Reboulet for discussions during the preparation of the paper. I am grateful to L\'aszl\'o Lempert for his remarks on the draft of the paper, especially the suggestion on Lemma \ref{lemma on T}. Thanks are also due to Tam\'as Darvas for his interest in the paper.

\section{Lemmas}\label{sec lemma}

We begin with a lemma that guarantees that the function $h$ in Definition \ref{super} is always upper semicontinuous.

\begin{lemma}\label{upper lemma}
    Let $u$ be an upper semicontinuous function on $D\times X$. The function $h(z):=\sup_{x\in X}u(z,x)$ is upper semicontinuous in $D$.
\end{lemma}
\begin{proof}
    Suppose that $\limsup_{z_0} h(z)> h(z_0)$ for some $z_0$ in $D$. We can find a positive constant $c$ such that $\limsup_{z_0} h(z)> h(z_0)+c$. For any positive integer $n$, there exists $z_n\in B(z_0, 1/n)$ such that 
    \begin{equation}
     h(z_n)> h(z_0)+c.    
    \end{equation}
     Meanwhile, $h(z_n)=\sup_{x\in X}u(z_n,x)=u(z_n,x_n)$ for some $x_n \in X$ because $u$ is upper semicontinuous and $X$ is compact. Using the compactness of $X$, there exists $\tilde{x}\in X$ such that some subsequence of $x_n$ converges to $\tilde{x}$ (we still denote the subsequence by $x_n$). By the upper semicontinuity of $u$, we have $\limsup_{(z_0,\tilde{x})} u\leq u(z_0,\tilde{x})\leq h(z_0)$. Therefore, there exists some $r>0$ such that $h(z_0)+c/2> \sup_{B((z_0,\tilde{x}),r )}u$. But for $n$ large, $$\sup_{B((z_0,\tilde{x}),r )}u\geq u(z_n,x_n)=h(z_n)>h(z_0)+c.$$ 
We get $h(z_0)+c/2>h(z_0)+c$, a contradiction. 

Alternatively, let $\pi_1:D\times X \to D$ be the projection. If $c\in \mathbb{R}$ then $\{z\in D: h(z)\geq c\}=\pi_1\{(z,x):u(z,x)\geq c\}$ is closed because $\{(z,x):u(z,x)\geq c\}$ is closed and $\pi_1$ is proper.
\end{proof}

\begin{lemma}\label{lemma on T}
Let $T$ be a biholomorphic map from $\tilde{D}$ to $D$ where $\tilde{D}$ and $D$ are open sets in $\mathbb{C}^m$. The following are equivalent 
\begin{enumerate}
    \item \label{statement 1} For any subharmonic function $h$ in an open subset of $D$, the composition $h\circ T$ is subharmonic.
    \item \label{statement 2}
    For any $z_0\in \tilde{D}$, if we denote by $A$ the Jacobian matrix $T'(z_0)$, then $\text{Tr}A^*MA\geq 0$ for any $m\times m$ Hermitian matrix $M$ with $\text{Tr}M\geq 0$.
    
    \item \label{statement 3} The Jacobian matrix $T'$ satisfies $T'(T')^{*}=a I_m$ for some positive function $a$ on $\tilde{D}$ where $I_m$ is the $m$-by-$m$ identity matrix and $(T')^{*}$ is the conjugate transpose of $T'$.
     
\end{enumerate}    
\end{lemma}

\begin{proof}
The equivalence between Statement \ref{statement 1} and Statement  \ref{statement 2} can be verified by straightforward computation. We now assume Statement \ref{statement 2} and try to prove Statement \ref{statement 3}. Replacing $M$ by $-M$, we see that $\text{Tr}A^*MA \leq  0$ if $M$ is Hermitian and $\text{Tr}M\leq 0$. So, the kernel of the linear functional $M\mapsto \text{Tr}A^*MA$ contains the kernel of the linear functional $M\mapsto \text{Tr}M$. This implies that the former functional is a constant multiple of the latter, and the constant is $\frac{1}{m}\text{Tr}A^*A$ which can be computed by setting $M=I_m$. Therefore, 
\begin{equation}\label{555}
\text{Tr}A^*MA=\frac{1}{m}(\text{Tr}A^*A) (\text{Tr}M), \text{ for any Hermitian $M$}.    
\end{equation}
Let $U$ be a unitary matrix such that $UAA^*U^*=\Delta$ where $\Delta$ is a diagonal matrix with diagonals $\{d_j\}_{j=1\sim m}$. Since $\text{Tr}A^*MA= \text{Tr}UAA^*U^*UMU^*=\text{Tr}\Delta UMU^*$ and $(\text{Tr}A^*A) (\text{Tr}M)=(\text{Tr}UAA^*U^*)( \text{Tr}UMU^*)=(\text{Tr}\Delta)( \text{Tr}UMU^*)$, formula (\ref{555}) becomes 
\begin{equation}
    \text{Tr}\Delta UMU^*=\frac{1}{m}(\text{Tr}\Delta)( \text{Tr}UMU^*), \text{ for any Hermitian $M$}.
\end{equation}
By choosing $UMU^*$ to be the diagonal matrix with the $k$-th diagonal equal to one and zero for other diagonals, we get $d_k=\sum_j d_j/m$. So, the $d_k$ are equal among themselves, say equal to $d$. Hence $AA^*=U^*\Delta U=d I_m$. After varying $z_0$, we get Statement \ref{statement 3}. The implication from Statement \ref{statement 3} to Statement \ref{statement 2} can be verified easily and we skip the details.    
\end{proof}

If $T$ satisfies the properties in Lemma \ref{lemma on T}, then so does $T^{-1}$ by using Statement \ref{statement 3}. When $m=1$, any biholomorphic $T$ satisfies the properties in Lemma \ref{lemma on T}. Another example that we will use later is given by $T(z_1,\ldots,z_m)=(\lambda_1z_1,\ldots, \lambda_m z_m)$
 with $|\lambda_j|=1$.

\begin{lemma}\label{lemma T}
Let $T$ be a biholomorphic map from $\tilde{D}$ to $D$ where $\tilde{D}$ and $D$ are open sets in $\mathbb{C}^m$. Assume $T$ satisfies the properties in Lemma \ref{lemma on T}. If $u$ is $\omega$-subharmonic on graphs in $D\times X$, then $u(T(z),x)$ is $\omega$-subharmonic on graphs in $\tilde{D}\times X$.

Moreover, if $u$ is $\omega$-harmonic on graphs in $D\times X$, then $u(T(z),x)$ is $\omega$-harmonic on graphs in $\tilde{D}\times X$.     
\end{lemma}

\begin{proof}
 Let $f$ be a holomorphic map from an open subset of $\tilde{D}$ to $X$. We want to show that $u(T(z),f(z))+\psi(f(z))$ is subharmonic where $\psi$ is a local potential of $\omega$. In fact, the function
\begin{equation}
u(T(z),f(z))+\psi(f(z))=u(\zeta,f(T^{-1}(\zeta)))+\psi(f(T^{-1}(\zeta))) 
\end{equation}
is subharmonic in $\zeta$ by $\omega$-subharmonicity of $u$ and it is also subharmonic in $z$ by the assumption on $T$ (we use Statement \ref{statement 1} in Lemma \ref{lemma on T}).

For the second part, we only need to show that $u(T(z),x)$ is $\omega$-superharmonic on graphs in $\tilde{D}\times X$. Let $v$ be $\omega$-subharmonic on graphs in $U\times X$ with $U$ an open subset of $\tilde{D}$. We want to show that $h(z)=\sup_{x\in X}v(z,x)-u(T(z),x)$ is subharmonic in $U$. In fact,
\begin{equation}
    h(z)=\sup_{x\in X}v(z,x)-u(T(z),x)=\sup_{x\in X}v(T^{-1}(\zeta),x)-u(\zeta,x)
\end{equation}
is subharmonic in $\zeta$ since $v(T^{-1}(\zeta),x)$ is $\omega$-subharmonic on graphs in $T(U)\times X$ by the first part and $u$ is $\omega$-superharmonic on graphs. By the assumption on $T$, the function $h(z)$ is subharmonic in $z$. 
\end{proof}

In \cite{Wu23}, the bounded open set $D\subset \mathbb{C}^m $ is assumed to be smooth strongly pseudoconvex (that is, a smooth strongly plurisubharmonic defining function for $D$). This assumption is needed only in \cite[Lemma 3.2]{Wu23} for the construction of barriers. In the next lemma, we show that the assumption can be relaxed to regular domains ($D$ being regular means that continuous boundary data always have harmonic extension to $D$).
\begin{lemma}\label{barrier}
    Let $D$ be a regular bounded open set in $\mathbb{C}^m$. Let $\nu$ be a continuous map from $\partial D\times X $ to $ \mathbb{R}$ such that $\nu(z,\cdot)\in \text{PSH}(X,\omega)$ for $z\in \partial D$, and denote by $V$ the sup of the Perron family 
\begin{align*}
G_\nu=\{u \in\textup{usc($D\times X$)}:& \textup{ $u$ is $\omega$-subharmonic on graphs}, \textup{and }\limsup_{D \ni z \to \zeta\in \partial D}u(z,x)  \leq \nu(\zeta,x)\}.
\end{align*}
Then $$\lim_{(z,x)\to(z_0,x_0)\in \partial D\times X} V(z,x)=\nu(z_0,x_0).$$
Moreover, $V$ is continuous on $D\times X$ and $V\in G_\nu$.

\end{lemma}

\begin{proof}
For $x\in X$, denote by $H_x(z)$ the function that is harmonic in $D$ and continuous in $\overline{D}$ with $H_x(\zeta)=\nu(\zeta,x)$ for $\zeta\in \partial D$. By the maximum principle of harmonic functions and the continuity of $\nu$, one can show that given $\varepsilon>0$, there exists $\delta>0$ such that $|H_x(z)-H_{x'}(z)|<\varepsilon$ for $d_\omega(x,x')<\delta$ and any $z\in \overline{D}$. Using the continuity of $H_x(z)$ in $z$ and the compactness of $X$, one can show that given $\varepsilon>0$, there exists $\delta'>0$ such that $|H_x(z)-H_x(z')|<\varepsilon$ for $|z-z'|<\delta'$ and any $x\in X$. Therefore, $H_x(z)$ is continuous jointly in $(z,x)\in \overline{D}\times X$.

For $u\in G_\nu$, $z\mapsto u(z ,x)$ is subharmonic by \cite[Lemma 3.1]{Wu23}. So $u(z,x)\leq H_x(z)$ and $V(z,x)\leq H_x(z)$ for any $(z,x)\in D\times X$. We have $\limsup_{(z,x)\to(z_0,x_0)\in \partial D\times X} V(z,x)\leq \nu(z_0,x_0)$.

For the other direction, we fix $z_0\in \partial D$ and define $P(\zeta)=\inf_{x\in X}\nu(\zeta,x)-\nu(z_0, x)$ for $\zeta \in 
 \partial D$. By Lemma \ref{upper lemma}, the function $P(\zeta)$ is continuous. Let $H(z)$ be the harmonic extension of $P(\zeta)$ to $\overline{D}$. The function $H(z)+\nu(z_0,x)$ is $\omega$-subharmonic on graphs and satisfies $\limsup_{z\to \zeta\in \partial D}H(z)+\nu(z_0,x)\leq \nu(\zeta,x)$ for any $x\in X$. Therefore, $H(z)+\nu(z_0,x)\leq V(z,x)$ and $\nu(z_0,x_0)\leq \liminf_{(z,x)\to(z_0,x_0)\in \partial D\times X} V(z,x)$.

The moreover part is proved in the same way as in \cite[Corollary 3.3]{Wu23}.
    
\end{proof}

\section{Proof of Theorem \ref{characterize}}\label{section charac}

In this section, we show that solutions of the WZW equation can be characterized by $\omega$-harmonicity on graphs. We will use the following formula from \cite{Wu23}. Suppose $u$ is a $C^2$ function on $D\times X$ and $\psi$ is a local potential of $\omega$. Consider the complex Hessian of $u+\psi$ with respect to a fixed coordinate $z_j$ in $D$ and local coordinates $x$ in $X$ where $\psi$ is defined
\begin{equation}\label{matrix}
\left (
\begin{array}{cccc}
(u+\psi)_{z_j\bar{z}_j}
& (u+\psi)_{z_j\bar{x}_1} & \cdots &  (u+\psi)_{z_j\Bar{x}_n}\\
(u+\psi)_{x_1\bar{z}_j} & (u+\psi)_{x_1\bar{x}_1} & \cdots & (u+\psi)_{x_1\bar{x}_n}\\
 \vdots & \vdots & \ddots & \vdots \\
(u+\psi)_{x_n\bar{z}_j}& (u+\psi)_{x_n\bar{x}_1} &\cdots & (u+\psi)_{x_n\bar{x}_n} 
\end{array}
\right );
\end{equation}
we will denote this matrix by $(u+\psi)_j$. Then
\begin{equation}\label{invariant}
\begin{aligned}
  &(i\partial \bar{\partial}u+\pi^*\omega)^{n+1}\wedge(i\sum_{j=1}^m dz_j\wedge d\bar{z}_j)^{m-1}\\
  =&(n+1)! (m-1)!\sum^m_{j=1}\det (u+\psi)_j  \big(\bigwedge^m_{k=1} i dz_k\wedge d\Bar{z}_k\wedge \bigwedge^n_{l=1} i dx_l\wedge d\Bar{x}_l\big).  
\end{aligned}
\end{equation}

The following lemma is implicit in the proof of \cite[Lemma 4.1]{Wu23} already.
\begin{lemma}\label{min lem}
Assume that $u$ is a $C^2$ function on $D\times X$ and $\omega +i\partial\Bar{\partial}u(z,\cdot)>0$ on $X$ for all $z\in D$. Let $f$ be any holomorphic function from an open subset of $D$ to $X$. We have
\begin{equation}\label{17}
\Delta(\psi(f(z))+u(z,f(z)))\geq \sum_j\frac{\det(u+\psi)_j}{\det (\psi_{\mu\Bar{\lambda}}+u_{\mu\Bar{\lambda}})} 
\end{equation}
where $\psi$ is a local potential of $\omega$, and $u_{\mu \bar{\lambda}}=\partial^2 u/\partial x_\mu \partial \bar{x}_\lambda$ and  $\psi_{\mu \bar{\lambda}}=\partial^2 \psi/\partial x_\mu \partial \bar{x}_\lambda$ with $x_\mu, x_\lambda$ local coordinates in $X$.

Moreover, fixing $(z_0,x_0)\in D\times X$, we can find $f$ with $f(z_0)=x_0$ such that the equality holds at $(z_0,x_0)$,
\begin{equation}
\Delta(\psi(f(z))+u(z,f(z)))|_{z_0}= \sum_j\frac{\det(u+\psi)_j}{\det (\psi_{\mu\Bar{\lambda}}+u_{\mu\Bar{\lambda}})}\big|_{(z_0,x_0)}. 
\end{equation}

\end{lemma}

\begin{proof}[Proof of Lemma \ref{min lem}]

We simply compute
\begin{align*}
&\Delta(\psi(f(z))+u(z,f(z)))=\\
    &\sum_{i ,\lambda, \mu}\psi_{ \mu\Bar{\lambda}}\frac{\partial f^{\mu}}{\partial z_i}\frac{\partial \Bar{f^{\lambda}}}{\partial \Bar{z_i}}+\sum_i u_{i\Bar{i}}+\sum_{i ,\lambda}u_{i \Bar{\lambda}}\frac{\partial \Bar{f^{\lambda}}}{\partial \Bar{z_i}}+\sum_{i,\mu}u_{\Bar{i}\mu}\frac{\partial f^{\mu}}{\partial z_i}+\sum_{i, \lambda, \mu}u_{ \mu\Bar{\lambda}}\frac{\partial f^{\mu}}{\partial z_i}\frac{\partial \Bar{f^{\lambda}}}{\partial \Bar{z_i}} 
\end{align*}
where $f^\mu$ is the $\mu$-th component of $f$. If we denote the matrix $(\psi_{\mu\Bar{\lambda}}+u_{\mu\Bar{\lambda}})$ by $A$ and the column vector $(u_{i\bar{\lambda}})$ by $B_i$, then the above expression is the same as 
\begin{equation}
    \sum_i \big( \langle A\frac{\partial f}{\partial z_i}, \frac{\partial f}{\partial z_i}\rangle+\langle B_i,\frac{\partial f}{\partial z_i}  \rangle+\overline{\langle B_i,\frac{\partial f}{\partial z_i}  \rangle}+u_{i\bar{i}} \big),
\end{equation}
where the angled inner product is the usual Euclidean inner product and $\partial f/\partial z_i$ is the column vector $(\partial f^{\mu}/\partial z_i)$. The matrix form can be further written as 
\begin{equation}\label{15}
    \sum_i \big(\|\sqrt{A}\frac{\partial f}{\partial z_i}+\sqrt{A}^{-1}B_i\|^2-\|\sqrt{A}^{-1}B_i\|^2+u_{i\Bar{i}}\big).
\end{equation}
 Notice that 
 \begin{equation}\label{16}
 \begin{aligned}
 \sum_i (-\|\sqrt{A}^{-1}B_i\|^2+u_{i\Bar{i}})&=\sum_i(u_{i\Bar{i}}-\langle A^{-1}B_i,B_i\rangle)=\sum_i(u_{i\Bar{i}}-\sum_{\lambda, \mu}u_{i\bar{\lambda}}(\psi+u)^{\bar{\lambda} \mu}u_{\Bar{i}\mu})\\
 &=\sum_i\frac{\det(u+\psi)_i}{\det (\psi_{\mu\Bar{\lambda}}+u_{\mu\Bar{\lambda}})}, 
 \end{aligned}
 \end{equation}
where the last equality can be deduced from Schur's formula for determinants of block matrices (for details see line 5, page 352 in \cite{Wu23}). From (\ref{15}) and (\ref{16}), we obtain (\ref{17}).

For a given point $(z_0, x_0)\in D\times X$, we may assume $(z_0, x_0)=(0,0)$ and consider $f(z)=\sum_i (-A^{-1}B_i)|_{(z_0,x_0)} z_i$. Evaluating at $(z_0, x_0)$, we see that the first term in (\ref{15}) is zero since $\partial f/\partial z_i (z_0)= -A^{-1}B_i (z_0,x_0)$. Hence,
$$\Delta(\psi(f(z))+u(z,f(z)))|_{z_0}=\sum_i\frac{\det(u+\psi)_i}{\det (\psi_{\mu\Bar{\lambda}}+u_{\mu\Bar{\lambda}})}|_{(z_0,x_0)}.$$
\end{proof}

By Lemma \ref{min lem}, we immediately have the following (see \cite[Lemma 4.1]{Wu23}).
\begin{lemma}\label{+ det}
  Suppose $u$ is a $C^2$ function on $D\times X$ and $\omega +i\partial\Bar{\partial}u(z,\cdot)>0$ on $X$ for all $z\in D$.  Then $u$ is $\omega$-subharmonic on graphs if and only if $$(i\partial \bar{\partial}u+\pi^*\omega)^{n+1}\wedge(i\sum_{j=1}^m dz_j\wedge d\bar{z}_j)^{m-1} \geq 0.$$
\end{lemma}
On the other hand, for $\omega$-superharmonicity on graphs, we have the following partial result.

\begin{lemma}\label{- det}
    Suppose $u$ is a $C^2$ function on $D\times X$ and $\omega +i\partial\Bar{\partial}u(z,\cdot)>0$ on $X$ for all $z\in D$.  Then $u$ is $\omega$-superharmonic on graphs if  $$(i\partial \bar{\partial}u+\pi^*\omega)^{n+1}\wedge(i\sum_{j=1}^m dz_j\wedge d\bar{z}_j)^{m-1} \leq 0.$$
\end{lemma}

\begin{proof}
    Let $v$ be defined on $U\times X$ with $U$ an open subset of $D$ and $v$ be $\omega$-subharmonic on graphs. Define $h(z)=\sup_{x\in X}v(z,x)-u(z,x)$. We want to show that $h(z)$ is subharmonic in $U$. By Lemma \ref{upper lemma}, $h$ is upper semicontinuous.

    Fixing $z_0$ in $U$, we have $h(z_0)=v(z_0,x_0)-u(z_0,x_0)$ for some $x_0$ in $X$ because $X$ is compact and $v-u$ is upper semicontinuous.
Using Lemma \ref{min lem}, we can find a holomorphic function $f$ from an open subset of $D$ to $X$ with $f(z_0)=x_0$  satisfying
\begin{equation}\label{small}
\Delta\big(u(z,f(z))+\psi(f(z))\big)|_{z_0}= \left. \sum_{j=1}^m\frac{\det(u+\psi)_j}{\det (\psi_{\mu\Bar{\lambda}}+u_{\mu\Bar{\lambda}})}\right \vert_{(z_0,x_0)}\leq 0; 
\end{equation}
the inequality is due to our assumption that $(i\partial \bar{\partial}u+\pi^*\omega)^{n+1}\wedge(i\sum_{j=1}^m dz_j\wedge d\bar{z}_j)^{m-1} \leq 0$ and formula (\ref{invariant}). 

According to the definition of $h$, we have
\begin{equation}\label{321}
 h(z)\geq v(z, f(z))-u(z,f(z))   
\end{equation}
with equality at $z_0$. We denote by $\fint_{B(z_0,r)}h(z)$ the average of $h$ over a ball centered at $z_0$ with radius $r$ in $\mathbb{C}^m$. From the inequality (\ref{321}), we see that 
\begin{align*}
 \fint_{B(z_0,r)}h(z)-h(z_0)\geq& \big(\fint_{B(z_0,r)} v(z, f(z))-u(z,f(z))\big) -   \big( v(z_0, x_0)-u(z_0,x_0)\big)\\
=&\big(\fint_{B(z_0,r)} v(z, f(z))+\psi(f(z))-u(z,f(z))-\psi(f(z))\big)\\ &-   \big( v(z_0, x_0)+\psi(x_0)-u(z_0,x_0)-\psi(x_0)\big)\\
\geq& -\big( \fint_{B(z_0,r)} u(z,f(z))+\psi(f(z))   \big)+u(z_0,x_0)+\psi(x_0), 
\end{align*}
where the last inequality comes from the fact that $v$ is $\omega$-subharmonic on graphs and hence $v(z,f(z))+\psi(f(z))$ is subharmonic. As a consequence, 
\begin{align*}
    &\liminf_{r\to 0} \frac{1}{r^2} \big(\fint_{B(z_0,r)} h(z)-h(z_0)\big)\\
    \geq& \liminf_{r\to 0} -\frac{1}{r^2}  \big(\fint_{B(z_0,r)} 
    u(z,f(z))+\psi(f(z))   -u(z_0,x_0)-\psi(x_0)
    \big)\\
    =&-\Delta \big( u(z,f(z))+\psi(f(z)) \big)| _{z_0}\geq 0,
\end{align*}
where the last inequality comes from (\ref{small}). Hence, $h$ is subharmonic by \cite{Saks} or \cite[Lemma 11.2]{CoifmanSemmes}, and $u$ is $\omega$-superharmonic on graphs.
\end{proof}

Let us now prove Theorem \ref{characterize}:

\begin{proof}[Proof of Theorem \ref{characterize}]
If $u$ solves the WZW equation, then by Lemma \ref{+ det} and Lemma \ref{- det}, $u$ is $\omega$-harmonic on graphs.

Conversely, assume $u$ is $\omega$-harmonic on graphs in $D\times X$. Suppose at a point $p$ in $D\times X$,
$$(i\partial \bar{\partial}u+\pi^*\omega)^{n+1}\wedge(i\sum_{j=1}^m dz_j\wedge d\bar{z}_j)^{m-1} >0 .$$
By formula (\ref{invariant}), we see $\sum_j \det(u+\psi)_j$ is positive at $p$. Since $u$ is $C^2$, there exists a compact neighborhood $N$ of $p$ in $D\times X$ where $\sum_j \det(u+\psi)_j$ is positive. Choose a smooth function $\rho\geq 0$ supported in $N$ but not identically zero such that $\sum_j \det(\rho+u+\psi)_j>0$ in $N$. By Lemma \ref{+ det}, the function $\rho+u$ is still $\omega$-subharmonic on graphs in $D\times X$ (this is where we use $\omega$-subharmonicity of $u$). However, since $u$ is $\omega$-superharmonic on graphs, the function $z\mapsto \sup_{x\in X} (\rho+u-u)(z,x)=\sup_{x\in X} \rho(z,x)$ is subharmonic in $D$. But the function $z\mapsto \sup_{x\in X}\rho(z,x)  $ is zero 
near the boundary $\partial D$, and by the maximum principle we have  $z\mapsto \sup_{x\in X} \rho(z,x) \leq 0 $, a contradiction. Therefore, over $D\times X$ $$(i\partial \bar{\partial}u+\pi^*\omega)^{n+1}\wedge(i\sum_{j=1}^m dz_j\wedge d\bar{z}_j)^{m-1} \leq 0 .$$ The opposite inequality also holds by Lemma \ref{+ det}, so $u$ solves the WZW equation.
\end{proof}

Lemma \ref{- det} is only a partial result, similar to \cite[Theorem 4.2 (b)]{Rochberg} and \cite[Theorem 15.4 (b)]{CoifmanSemmes} that studied interpolation of norms; the full equivalence of their case was proved by Slodkowski in \cite[Theorem 6.6]{SlodI} and in \cite{SlodIII}. In the study of interpolation of norms, one can simply take the dual of a norm, but in our case, it is not obvious how to take dual on a function defined on $D\times X$. Nevertheless, careful examination suggests some type of Legendre transform. We give a heuristic computation below.

The idea coming from \cite{CoifmanSemmes,legendredualBCEKR} is to replace the inner product in the classical Legendre transformation by a local potential of $\omega$. To that end, it seems necessary to assume that $\omega$ is real analytic, so locally $\omega=i\partial\bar{\partial }\psi$ for some real analytic $\psi$, and we can consider the polarization $\psi_\mathbb{C}$ of $\psi$. If $B$ a ball  centered at 0 in $\mathbb{C}^n$  is a coordinate system of $X$ and $\psi(x)=\sum c_{\alpha\beta}x^\alpha\bar{x}^\beta$ in $B$, then $\psi_\mathbb{C}(x,y)=\sum c_{\alpha \beta}x^\alpha \bar{y}^\beta$ in some neighborhood of the diagonal in $B\times B$. In order to have a transformation that is defined globally on the manifold $X$, we will use the Calabi diastasis function 
\begin{equation}
    D_\omega(x,y)=\psi(x)+\psi(y)-2\text{Re}\psi_\mathbb{C}(x,y).
\end{equation}
For a function $u:D\times X\to \mathbb{R}$, we define the Legendre transform 
\begin{equation}\label{L transform}
    \hat{u}(z,y)=\sup_{x\in X} -D_\omega(x,\bar{y})-u(z,x).
\end{equation}
Note that $\hat{u}$ is a function on $D\times X$ and the transformation is performed only in the $x$ variable. It is not hard to see that the K\"ahler form $i\partial\bar{\partial}\psi(\bar{x})$ is globally defined on $X$, and we will denote it by $\tilde{\omega}$. (The extra complex conjugation $\bar{y}$ in (\ref{L transform}) has appeared in \cite{Rochberg,SlodI, CoifmanSemmes} when they studied duality for norms).

We have the following duality: 
\begin{equation}\label{L dual}
  \text{if $u$ is $\omega$-superharmonic on graphs, then $\hat{u}$ is $\tilde{\omega}$-subharmonic on graphs.}  
\end{equation}
Indeed, letting $f$ be a holomorphic map from some open subset of $D$ to $X$, we want to show that the following function is subharmonic
\begin{align*}
  \hat{u}(z,f(z))+\psi(\overline{f(z)})=\sup_{x\in X} \big( 2\text{Re}\psi_\mathbb{C}(x,\overline{f(z)})-\psi(x)-\psi(\overline{f(z)})  -u(z,x)\big)  +\psi(\overline{f(z)})\\
  =\sup_{x\in X} 2\text{Re}\psi_\mathbb{C}(x,\overline{f(z)})-\psi(x)  -u(z,x).  
\end{align*}
This is true since $(z,x)\mapsto 2\text{Re}\psi_\mathbb{C}(x,\overline{f(z)})-\psi(x)$ is $\omega$-subharmonic on graphs and $u(z,x)$ is $\omega$-superharmonic on graphs. That $(z,x)\mapsto 2\text{Re}\psi_\mathbb{C}(x,\overline{f(z)})-\psi(x)$ is $\omega$-subharmonic on graphs is because for a holomorphic map $g(z)$ from some open subset of $D$ to $X$, the function $$2\text{Re}\psi_\mathbb{C}((g(z)),\overline{f(z)})-\psi(g(z))+\psi(g(z))=2\text{Re}\psi_\mathbb{C}((g(z)),\overline{f(z)})$$
is subharmonic by $\psi_\mathbb{C}(x,y)=\sum c_{\alpha \beta}x^\alpha \bar{y}^\beta$.

\begin{conjecture}
    We conjecture that the Legendre transform will also turn $\omega$-subharmonicity to $\tilde{\omega}$-superharmonicity, and this fact along with (\ref{L dual}) can be used to prove the converse of Lemma \ref{- det}.
\end{conjecture}

\section{Proof of Theorem \ref{sup is har}}\label{section sup is har}

Let $D$ be a regular bounded open set in $\mathbb{C}^m$. Let $\nu$ be a continuous map from $\partial D\times X $ to $ \mathbb{R}$ such that $\nu(z,\cdot)\in \text{PSH}(X,\omega)$ for $z\in \partial D$. Recall the Perron family 
\begin{align*}
G_\nu=\{u \in\textup{usc($D\times X$)}:& \textup{ $u$ is $\omega$-subharmonic on graphs}, \textup{and }\limsup_{D \ni z \to \zeta\in \partial D}u(z,x)  \leq \nu(\zeta,x)\}.
\end{align*} 
Denote by $V$ the sup of the Perron family $G_\nu$. In this section, we assume $\omega$ is in an integral class, so $V$ is in $G_\nu$, $V$ attains the boundary data $\nu$, and $V$ extends  continuously on $\overline{D}\times X$ according to  Lemma \ref{barrier}. The goal of the section is to show that $V$ is $\omega$-harmonic on graphs.

\begin{lemma}\label{max principle}
Let $u$ be $\omega$-subharmonic on graphs in $D\times X$ and upper semicontinuous on $\overline{D}\times X$. The function $$h(z):=\sup_{x\in X} u(z,x)-V(z,x)$$ satisfies $\sup_D h \leq \sup_{\partial D}h$.  
\end{lemma}
\begin{proof}
If we denote $\sup_{\partial D}h$ by $M$, then $ u(z,x)-V(z,x)\leq M$ for $(z,x)\in \partial D\times X$. Since $V=\nu$ on $\partial D \times X$, we have   $u(z,x)-\nu(z,x)\leq M$ for $(z,x)\in\partial D\times X$. The function $u(z,x)-M$ is $\omega$-subharmonic on graphs, and hence it is in $G_{\nu}$. As a result, $u(z,x)-M\leq V(z,x)$ on $D\times X$ and $\sup_D h \leq M$.  
\end{proof}


Let $U$ be a regular open subset of $D$. Since the sup $V$ is in $G_\nu$, $V$ is $\omega$-subharmonic on graphs and by \cite[Lemma 3.1]{Wu23}, for fixed $z$, $V(z, \cdot)$ is in $\text{PSH}(X,\omega)$, so we can use $V|_{\partial U \times X}$ as boundary data and consider the Perron family $G_{V|_{\partial U\times X}}$. We have the following reiteration lemma.

\begin{lemma}\label{reiteration}
The sup of $G_{V|_{\partial U\times X}}$ which we denote by $\mathcal{V}$ satisfies $\mathcal{V}=V$ on $U\times X$.
 
\end{lemma}

\begin{proof}
    For any $u\in G_\nu$, the restriction $u|_{U\times X}$ is in $G_{V|_{\partial U\times X}}$, so $u|_{U \times X}\leq \mathcal{V}$. Since $V$ is in $G_\nu$, we have $V\leq \mathcal{V}$ on $U\times X$. 
Conversely, consider $u_0\in G_{V|_{\partial U\times X}}$ and define 
\begin{equation}
u=\begin{cases}
  V \text{ on } (D\setminus U)\times X, \\ 
\max(u_0, V)\text{ on } U\times X.
\end{cases}    
\end{equation}
Using the fact $V$ and $u_0$ are both $\omega$-subharmonic on graphs, it is straightforward to check $u$ is in $G_\nu$. Therefore, $u\leq V$ and $u_0\leq V$ on $U\times X$, hence $\mathcal{V} \leq V$ on $U\times X$.
\end{proof}

\begin{proof}[Proof of Theorem \ref{sup is har}]

We already know that $V$ is in $G_\nu$, so $V$ is $\omega$-subharmonic on graphs. For $\omega$-superharmonicity, we let $u$ be $\omega$-subharmonic on graphs in $U\times X$ with $U$ an open subset of $D$. We want to show that  
$h(z):=\sup_{x\in X} u(z,x)-V(z,x)$ is subharmonic in $U$. By Lemma \ref{upper lemma}, $h$ is upper semicontinuous.

It suffices to show that, for any ball $B$  with $\overline{B}\subset U$ and any harmonic function $g(z)$ in $B$ continuous up to $\overline{B}$, the function $h-g$ satisfies the maximum principle on $B$: $\sup_B (h-g) \leq \sup_{\partial B}(h-g)$. It is clear that $u(z, x)-g(z)$ is $\omega$-subharmonic on graphs in $B\times X$ and upper semicontinuous on $\overline{B}\times X$. Meanwhile, Lemma \ref{reiteration} says that the Perron family with boundary data $V|_{\partial B \times X}$ still has the same sup $V$. Therefore, using Lemma \ref{max principle} on $\overline{B}\times X$, we see that  $$\sup_{x\in X}u(z, x)-g(z)-V(z,x)=h(z)-g(z)$$ satisfies the maximum principle on $B$. Therefore $h$ is subharmonic in $U$ and $V$ is $\omega$-superharmonic on graphs.    
\end{proof}

\section{Harmonic maps}\label{section harmonic}

Let us recall the setup in Subsection \ref{subsection har}. Let $L$ be a positive line bundle over $X$, and $h$ a positively curved metric on $L$ with curvature $\omega$. For a positive integer $k$ we denote by $\mathcal H_k$ the space of inner products on $H^0(X,L^k)$.

Let $D'$ be a regular bounded open set in $\mathbb{R}^m$. We denote the coordinates in $D'$ by $t=(t_1,\dots, t_m)\in  \mathbb{R}^m$. Let $\nu$ be a real-valued continuous function on $\partial D'\times X$ such that for fixed $t\in \partial D'$ the function $ \nu(t,\cdot)$ on $X$ is in $ \text{PSH}(X,\omega)$. 

We recall the Hilbert map and the Fubini--Study map. The Hilbert map $H_k: \text{PSH}(X,\omega)\to \mathcal H_k$ is $$H_k(\phi)(s,s)=\int_X h^k(s,s)e^{-k\phi}\omega^n, \text{ for $\phi\in \text{PSH}(X,\omega)$ and $s\in H^0(X,L^k)$}.$$ The Fubini--Study map $FS_k: \mathcal H_k\to \mathcal H_{\omega}$ is
\begin{equation*}
FS_k(G)(x)=\frac{1}{k}\log \sup_{s\in H^0(X,L^k),G(s,s)\leq 1}h^k(s,s)(x), \text{ for $G\in \mathcal H_k$ and } x\in X.
\end{equation*}
Refer to \cite{DLR,DarvasWu} for more details about the maps.

For the boundary data $H_k(\nu):\partial D'\to \mathcal H_k$, we know by \cite{Hamilton} that the following Dirichlet problem for the harmonic map equation has a unique solution $V^k:\overline{D'}\to \mathcal H_k $  
\begin{equation}
\begin{dcases}\label{harmonic matrix 1}
 \sum_{j=1}^m \frac{\partial}{\partial t_j}\big((V^k)^{-1}  \frac{\partial V^k}{\partial t_j}\big)=0 \\
V^k|_{\partial D'}=H_k(\nu).
\end{dcases}    
\end{equation}

\begin{proof}[Proof of Theorem \ref{thm weak harmonic}]

For the coordinates $t=(t_1,\dots, t_m)$ in $D'\subset \mathbb{R}^m$, we complexify the variable $t_j$ by adding a variable $\sqrt{-1}s_j$ with $s_j\in \mathbb{R}$ and introducing $$e^{t_j+\sqrt{-1}s_j}=\zeta_j\in \mathbb{C}$$ so that $t_j=\log |\zeta_j|$. We denote by $D\subset \mathbb{C}^m$ the corresponding set for the complex variables $(\zeta_1,\dots, \zeta_m)$; namely, $$D=\{(\zeta_1,\dots, \zeta_m)\in \mathbb{C}^m: (\log|\zeta_1|,\dots, \log|\zeta_m|)\in D' \}.$$ We define $\tilde{\nu}:\partial D\to \text{PSH}(X,\omega)$ by setting 
\begin{equation}
  \tilde{\nu}(\zeta_1,\dots, \zeta_m)=\nu(\log |\zeta_1|,\dots, \log |\zeta_m|).  
\end{equation}
Then as in \cite{Wu23}, we use the boundary data $\Tilde{\nu}$ to consider the Perron family 
\begin{align*}
G_{\Tilde{\nu}}:=\{u \in\text{usc($D\times X$) }:& \text{ $u$ is $\omega$-subharmonic on graphs}, \text{and }\limsup_{D \ni \zeta \to \eta\in \partial D}u(\zeta,x)  \leq \Tilde{\nu}(\eta,x)\}.
\end{align*}
A small change from the setup in \cite{Wu23} is that the metric we use on $D$ here is not Euclidean $\sum_j d\zeta_j\wedge d\bar{\zeta}_j$ but $\sum_j1/|\zeta_j|^2d\zeta_j\wedge d\bar{\zeta}_j$.

The upper envelope $\tilde{V}=\sup\{u: u \in G_{\tilde{\nu}}\}$ is continuous in $D\times X$, attains the boundary data $\tilde{\nu}$, and is $\omega$-harmonic on graphs in $D\times X$ by Theorem  \ref{sup is har} and Lemma \ref{barrier}. 

The boundary data $\Tilde{\nu}$ is rotationally invariant, namely, $\tilde{\nu}(\lambda_1\zeta_1,\dots, \lambda_m\zeta_m)=\tilde{\nu}(\zeta_1,\dots, \zeta_m)$ with $|\lambda_j|=1$. We claim that $\Tilde{V}$ is also rotationally invariant. In fact, $\tilde{V}(\lambda_1\zeta_1,\dots, \lambda_m\zeta_m)$ with $|\lambda_j|=1$ is $\omega$-harmonic on graphs by Lemma \ref{lemma T} and shares the same boundary data with $\tilde{V}(\zeta_1,\dots, \zeta_m)$. Hence by Theorem \ref{dist subhar}, we have 
$\tilde{V}(\lambda_1\zeta_1,\dots, \lambda_m\zeta_m)=\tilde{V}(\zeta_1,\dots, \zeta_m)$, as claimed.

One can also show the rotational invariance of $\tilde{V}$ through quantization/approximation. In fact, let us consider the boundary data $H_k(\Tilde{\nu}):\partial D \to \mathcal H_k$ which can be viewed as a Hermitian metric on the bundle $\partial D \times H^0(X, L^k) \to \partial D $.  There exists a unique smooth Hermitian metric $\tilde{V}^k$ on the bundle $D\times H^0(X, L^k)\to D$ that solves the Hermitian--Yang--Mills equation (\cite{Donaldson92,CoifmanSemmes})
\begin{equation}\label{HYM}
\begin{cases}
\Lambda \Theta(\tilde{V}^k)=0\\
\tilde{V}^k|_{\partial D}=H_k(\Tilde{\nu}),
\end{cases}    
\end{equation}
where $\Theta(\tilde{V}^k)$ is the curvature of the Hermitian metric $\tilde{V}^k$ and $\Lambda$ is the trace with respect to the metric $\sum_j1/|\zeta_j|^2d\zeta_j\wedge d\bar{\zeta}_j$ on $D$.
Moreover, $FS_k(\tilde{V}^k)$ converges to $\Tilde{V}$ uniformly on $D\times X$ by \cite[Theorem 1.2]{Wu23} (when applying this theorem, one has to take an extra duality on the Hermitian metrics). The boundary data $H_k(\Tilde{\nu})$ is rotationally invariant, and so is $\tilde{V}^k$ due to the uniqueness of the solutions of the Hermitian--Yang--Mills equation. Hence, the limit $ \lim_{k\to \infty} FS_k(\tilde{V}^k)= \Tilde{V}$ is rotationally invariant. 

One way or another, we can define a function $V$ on $D'\times X$ by setting $$V(t_1,\dots,t_m, x)= \Tilde{V}(\zeta_1,\dots, \zeta_m, x)$$ where $t_j=\log |\zeta_j|$. Since $\tilde{V}$ is continuous, so is $V$. Meanwhile, we know that $\Tilde{V}|_{\partial D}=\tilde{\nu}$, so  
$V|_{\partial D'}=\nu$. Also, $\tilde{V}(\zeta, \cdot)$ is in $\text{PSH}(X,\omega)$ for fixed $\zeta$ by \cite[
Lemma 3.1]{Wu23}, hence we see that $V(t,\cdot)$ is in $\text{PSH}(X,\omega)$ for fixed $t$.

So far we have proved the first part of Theorem \ref{thm weak harmonic}. Now we move on to the quantization. If we use some basis for $H^0(X,L^k)$ as a global frame for the bundle $D\times H^0(X, L^k)\to D$ and represent the Hermitian metric $\tilde{V}^k$ as matrix-valued functions, then the Hermitian--Yang--Mills equation (\ref{HYM}) is   
\begin{equation}\label{HYM1}
    \sum_{j=1}^m |\zeta_j|^2 \frac{\partial}{\partial \Bar{\zeta}_j}\big((\tilde{V}^k)^{-1}  \frac{\partial \tilde{V}^k}{\partial \zeta_j}\big)=0;    
\end{equation}
note again that the metric we use on $D$ is $\sum_j1/|\zeta_j|^2d\zeta_j\wedge d\bar{\zeta}_j$.
Since $\tilde{V}^k$ is rotationally invariant and can be viewed as a map from $D$ to $ \mathcal H_k$, it gives rise to a map $V^k: D'\to \mathcal H_k$  with $$V^k(t_1,\dots, t_m)=\tilde{V}^k(\zeta_1,\dots, \zeta_m) \text{ and } t_j=\log |\zeta_j|. $$
By the chain rule 
\begin{equation}\label{chain rule}
 \frac{\partial  \tilde{V}^k   }{\partial \zeta_j}=\frac{\partial V^k}{\partial t_j}\frac{1}{2\zeta_j} \text{  and  } \frac{\partial^2 \tilde{V}^k}{\partial \zeta_j \partial \bar{\zeta}_j  }=\frac{\partial^2 V^k}{\partial t_j^2}\frac{1}{4|\zeta_j|^2},
\end{equation}
equation (\ref{HYM1}) becomes
\begin{equation}\label{harmonic for matrix}
    \sum_{j=1}^m \frac{\partial}{\partial t_j}\big((V^k)^{-1}  \frac{\partial V^k}{\partial t_j}\big)=0    
\end{equation}
which is the harmonic map equation for maps from $D'$ to $ \mathcal H_k$, namely (\ref{harmonic matrix 1}). The uniform convergence
$\lim_{k\to \infty} FS_k(\tilde{V}^k)= \Tilde{V}$ on $D\times X$ thus translates to the uniform convergence
$\lim_{k\to \infty} FS_k(V^k)= V$ on $D'\times X$.

Finally, if $V$ or $\tilde{V}$ is $C^2$, then $\tilde{V}$ solves the WZW equation by \cite[Theorem 1.3]{Wu23}
\begin{equation}\label{wzw variant 1}
\sum^m_{j=1}|\zeta_j|^2\big(|\nabla\Tilde{V}_{\zeta_j}|_{\Tilde{V}}^2-2\Tilde{V}_{\zeta_j\bar{\zeta}_j}+i\{\Tilde{V}_{\bar{\zeta}_j},\Tilde{V}_{{\zeta_j}}\}_{\Tilde{V}}\big)=0. 
\end{equation}
Using the chain rule similar to (\ref{chain rule}), we turn equation (\ref{wzw variant 1}) to   
\begin{equation}\label{harmonic eq}
\sum^m_{j=1}|\nabla V_{t_j}|_{V}^2-2V_{t_jt_j}+i\{V_{t_j},V_{{t_j}}\}_{V}=0;
\end{equation}
since the Poisson bracket $\{V_{t_j},V_{{t_j}}\}_{V}$ is zero, the equation (\ref{harmonic eq}) is the harmonic map equation (see equation (\ref{har})). 
\end{proof}

From the proof of Theorem \ref{thm weak harmonic}, we see that the weak solution $V$ of the harmonic map equations into $\mathcal H_\omega$ is given by the sup of the Perron family 
\begin{align*}
G_{\Tilde{\nu}}=\{u \in\text{usc($D\times X$) }:& \text{ $u$ is $\omega$-subharmonic on graphs}, \text{and }\limsup_{D \ni \zeta \to \eta\in \partial D}u(\zeta,x)  \leq \Tilde{\nu}(\eta,x)\}.
\end{align*}
Actually a subfamily will yield the same sup. Indeed, let us consider the rotationally invariant Perron family $G^{R.I.}_{\Tilde{\nu}}$ that consists of $u\in G_{\Tilde{\nu}}$ which are invariant under rotation:
$u(\lambda_1\zeta_1,\dots, \lambda_m\zeta_m)=u(\zeta_1,\dots, \zeta_m)$ with $|\lambda_j|=1$. We claim that $$\sup \{u: u\in G^{R.I.}_{\Tilde{\nu}}\}=\sup \{u: u\in G_{\Tilde{\nu}}\}.$$ Let us denote $\sup \{u: u\in G_{\Tilde{\nu}}\}$ by $\tilde{V}$. By Lemma \ref{barrier}, the envelope $\tilde V$ is in $G_{\Tilde{\nu}}$, but $\tilde V$ is rotationally invariant as is shown in the proof of Theorem \ref{thm weak harmonic}, so $\tilde V$ is in $G^{R.I.}_{\Tilde{\nu}}$. Therefore, we see
$$\tilde V \leq \sup \{u: u\in G^{R.I.}_{\Tilde{\nu}}\}\leq \sup \{u: u\in G_{\Tilde{\nu}}\}=\tilde V.$$
Such a characterization of harmonic maps is perhaps not that satisfactory because we still have to use $\omega$-subharmonicity which involves holomorphic terms.

Similarly, the solution $V^k$ of the harmonic map equations (\ref{harmonic matrix 1}) into $\mathcal H_k$  can also be written as the sup of some rotationally invariant Perron family. In fact, The solution $\tilde{V}^k$ of the Hermitian--Yang--Mills equation (\ref{HYM}) is the sup of the Perron family
\begin{align*}
  G^{k}_{\tilde \nu}:=\{&D \ni \zeta \to U_\zeta \in \mathcal N_k \textup{ is subharmonic and }\\ 
  &\limsup_{D\ni \zeta \to \eta \in \partial D} U^2_\zeta(s) \leq H_k(\tilde{\nu}_\eta)(s,s) \textup{ for any $s\in H^0(X,L^k)$ }  \}, 
\end{align*}
where $\mathcal N_k$ is the set of norms on $H^0(X,L^k)$ and $U_\zeta$ is said to be subharmonic if $\log U_\zeta(f(\zeta))$ is subharmonic for any holomorphic section $f:W\subset D \to H^0(X,L^k)$.
The rotationally invariant subfamily $(G^{k}_{\tilde \nu})^{R.I.}$ will yield the same sup by the same argument.

\section{Appendix}\label{sec appendix}

In this appendix, we prove the claim in Subsection \ref{subsection har} that, when $m=1$ and $D'$ is the open interval $(0,1)$, $V$ satisfies $d(V(s),V(t))=|s-t|d(V(0),V(1))$. Indeed, $V$ as the sup of the Perron family is $\omega$-harmonic on graphs, and $V(0)$ as a function on $D'\times X$, independent of the variable in $D'$, is also  $\omega$-harmonic on graphs (as if extending across $D'$ trivially), so by Theorem \ref{dist subhar} the function $t\mapsto d(V(0),V(t))$ is convex. Similarly, $t\mapsto d(V(1),V(t))$ is also convex. As result,
\begin{equation}\label{convex}
    d(V(0),V(t))\leq td(V(0),V(1)) \text{ and }
    d(V(t),V(1))\leq (1-t)d(V(0),V(1)).
\end{equation}
Adding these two inequalities and using the triangle inequality for $d$, we get
$$d(V(0),V(1))\leq d(V(0),V(t))+d(V(t),V(1))\leq d(V(0),V(1))$$
which forces the equality sign in (\ref{convex}). Next, we fix $s_0$ and vary $t$ between $s_0$ and 1. Since $V$ and $V(s_0)$ are both $\omega$-harmonic on graphs, we get by convexity
\begin{equation}\label{convex 2}
    d(V(s_0),V(t))\leq \frac{t-s_0}{1-s_0}d(V(s_0),V(1)) \text{ and }
    d(V(t),V(1))\leq \frac{1-t}{1-s_0}d(V(s_0),V(1)).
\end{equation}
Adding these two inequalities, we get
\begin{equation*}
    d(V(s_0),V(1)) \leq d(V(s_0),V(t))+d(V(t),V(1))\leq d(V(s_0),V(1)).
\end{equation*}
So, (\ref{convex 2}) are actually equalities. The second equality in (\ref{convex}) and the first equality in (\ref{convex 2}) together give  
$d(V(s_0),V(t))= (t-s_0))d(V(0),V(1))$ as claimed.

\bibliographystyle{amsalpha}
\bibliography{dom}

\textsc{Erdős Center, HUN-REN Rényi Institute, Reáltanoda utca 14, H-1053, Budapest, Hungary}

\texttt{\textbf{wuuuruuu@gmail.com}}

\end{document}